\documentclass[a4paper,12pt, reqno]{amsart}
\usepackage{amssymb,amsfonts}
\usepackage[all,arc]{xy}
\usepackage{enumerate}
\usepackage{mathrsfs}
\usepackage[utf8]{inputenc}
\usepackage{amsmath,mathtools,tikz,graphicx,caption,subcaption}
\usepackage{xcolor}
\usepackage{amssymb}
\newtheorem{thm}{Theorem}[section]
\newtheorem{prop}[thm]{Proposition}

\theoremstyle{definition}
\newcommand{\Q}{\mathbb{Q}}

\newtheorem{theorem}{Theorem}[section]
\newtheorem{lemma}{Lemma}[section]
\newtheorem{example}{Example}[section]

\newtheorem{corollary}{Corollary}[section]
\newtheorem{definition}{Definition}[section]

\theoremstyle{remark}

\newtheorem*{acknowledgements*}{Acknowledgements}

\makeatletter
\def\blfootnote{\gdef\@thefnmark{}\@footnotetext}
\makeatother

\def\house#1{\setbox1=\hbox{$\,#1\,$}%
\dimen1=\ht1 \advance\dimen1 by 2pt \dimen2=\dp1 \advance\dimen2 by 2pt
\setbox1=\hbox{\vrule height\dimen1 depth\dimen2\box1\vrule}%
\setbox1=\vbox{\hrule\box1}%
\advance\dimen1 by .4pt \ht1=\dimen1
\advance\dimen2 by .4pt \dp1=\dimen2 \box1\relax}

\begin{document}

\title[A new characterization for the Lucas-Carmichael integers \\ and sums of base-$p$ digits]{A new characterization for the Lucas-Carmichael integers and sums of base-$p$ digits}
\author[Sridhar Tamilvanan and  Subramani Muthukrishnan] {Sridhar Tamilvanan $^{(1)}$ and Subramani Muthukrishnan $^{(2)}$}
\address{$^{(1)}$  Indian Institute of Information Technology, Design and Manufacturing, Kancheepuram, Chennai - 600127, INDIA}
\address{$^{(2)}$ Indian Institute of Information Technology, Design and Manufacturing, Kancheepuram, Chennai - 600127, INDIA}
\email{$^{(1)}$mat20d001@iiitdm.ac.in}
\email{$^{(2)}$subramani@iiitdm.ac.in}

\begin{abstract}
In this paper, we prove a  necessary and sufficient condition for the Lucas-Carmichael integers in terms of the sum of base-$p$ digits. We also study some interesting properties of such integers. Finally, we prove that there are infinitely many Lucas-Carmichael integers assuming the prime $k$-tuples conjecture.
\end{abstract}
\maketitle

\section{Introduction}

The classical Fermat's little theorem states that if $p$ is a prime number, then any positive integer $a$ satisfies $a^{p} \equiv a \pmod{p}.$ In particular, if $a$ is not divisible by $p$, we have $a^{p-1} \equiv 1 \pmod{p}.$ However, the converse of Fermat’s little theorem is not true. For an example, $561$ is the least composite integer that satisfies $a^{560} \equiv 1 \pmod{561}$ for every positive integer $a$ with $gcd(a, 561) = 1.$ Such integers are called Carmichael integers. In general, we have the following definition:

\begin{definition}
A composite positive integer $m$ is called a Carmichael number if the congruence $a^{m-1} \equiv 1 \pmod{m}$ holds for all integers $a$ co-prime to $m.$
\end{definition}

\noindent In 1899, A. Korselt \cite{Korselt 1} observed an important criterion for Carmichael numbers.

\begin{theorem}(Korselt's criterion). A composite number $m$ is a Carmichael number if and only if $m$ is square-free and every prime divisor $p$ of $m$ satisfies $p - 1$ $|$ $m - 1.$
\end{theorem}

\noindent
Later, R. D. Carmichael \cite{Carmichael 1,Carmichael 2} proved a few interesting properties for Carmichael numbers.

\begin{theorem}(R. D. Carmichael). Every Carmichael number $m$ is odd, square-free and has at least three prime factors. If $p$ and $q$ are prime divisors of $m,$ then
$$(i)~ p - 1~ |~ m - 1,~~~~~~ (ii) ~p - 1~ |~ \frac{m}{p} - 1~~~  \textrm{ and }~~~ (iii)~ p ~\nmid~ q - 1.$$
\end{theorem}

\noindent
In 1994, W. R. Alford, Andrew Granville, and Carl Pomerance \cite{Alford 1} solved the long-standing conjecture that the set of all Carmichael numbers is infinite. More precisely, they showed that if $C(x)$ denotes the number of Carmichael numbers less than $x,$ then $C(x) > x^{2/7}$ for sufficiently large $x.$
\medskip

\noindent
Recently, B. C. Kellner and J. Sondow \cite{Kellner and Sondow} derived a new characterization for Carmichael numbers as follows: For a positive integer $m,$ denote $S_p(m)$ as the sum of the base-$p$ digits of $m$. Then, $m$ is a Carmichael number if and only if it is square-free and each of its prime factors $p$ satisfies $S_p(m) \geq p$ and $S_p(m) \equiv 1 \pmod{p-1}.$ In particular, a primary Carmichael number $m$ is a Carmichael number that satisfies $S_p(m) = p$ for every prime factor $p$ of $m.$ B. C. Kellner and J. Sondow \cite{Kellner and Sondow} counted the Carmichael numbers and primary Carmichael numbers up to $10^{10}.$ In 2022, Wagstaff \cite{Samuel} proved that the prime $k$-tuples conjecture implies that there are infinitely many primary Carmichael numbers.
\medskip

\noindent
In this paper, we study a variation of Carmichael numbers motivated by Gordon's primality testing algorithm \cite{Gordon} as explained below. An elliptic curve $E$ over $\mathbb{Q}$ is a smooth projective curve that satisfies the Weierstrass equation $$E : Y^2 = X^3 + aX+b,$$ where $a, b \in \mathbb{Q}$ and discriminant $\Delta = 4a^3 + 27b^2 \neq 0.$  For an elliptic curve $E$ with complex multiplication by $\Q(\sqrt{-d}),$ let $P \in E(\Q)$ be a rational point of infinite order and $m$ be a positive integer such that $gcd(m, 6\Delta) = 1$ with $\Big(\frac{-d}{m}\Big) = -1,$ where $\Big(\frac{-d}{m}\Big)$ is the Jacobi symbol. If $m$ is a prime, then
$$[m + 1]P \equiv \mathcal{O} \pmod{m}.$$
If $m$ satisfies the above congruence, then $m$ is a probable prime by Gordon’s primality test. Also, $m$ is a composite number when $m$ does not satisfy the above congruence relation.

\begin{definition}
Let $m$ be a composite number and $E$ be a $CM$-elliptic curve. If $m$ satisfies the Gordon primality test, then $m$ is called an $E$-elliptic Carmichael number. A composite integer $m$ is said to be an elliptic Carmichael number if $m$ is an $E$-elliptic Carmichael number for every $CM$-elliptic curve $E.$
\end{definition}

\noindent
Ekstrom et al. \cite{EPT 1} computed the following smallest elliptic Carmichael number:
\begin{align*}
 617730918224831720922772642603971311 = p(2p + 1)(3p + 2),
\end{align*}
where $p = 468 686 771 783.$ Also, they proved the following Elliptic Carmichael condition.

\begin{theorem}(Elliptic Carmichael Condition).
Let $m$ be a square-free, composite positive integer with an odd number of prime factors. Moreover, let $\alpha  = 8 \cdot 3 \cdot 7 \cdot 11 \cdot 19 \cdot 43 \cdot 67 \cdot 163.$ Then $m$ is an elliptic Carmichael number if for each prime $p$ $|$ $m,$ we have $\alpha$ $|$ $p + 1$ and $p + 1$ $|$ $m + 1.$
\end{theorem}

\noindent
Observing the above elliptic Carmichael condition, the following Korselt-like criterion has been noted: $p + 1$ $|$ $m + 1$ whenever $p$ $|$ $m$, and thus the Lucas-Carmichael integers have been defined.

\begin{definition}
A Lucas-Carmichael integer is a square-free positive composite integer $m$ such that $p+1$ $|$ $m+1$ whenever $p$ $|$ $m$.
\end{definition}

\noindent
In 2018, Thomas Wright \cite{Thomas Wright} proved that there are infinitely many Lucas-Carmichael integers. In fact, he showed that if $\mathcal{N}(X)$ denotes the number of elliptic Carmichael numbers up to $X,$ then there exists a constant $K>0$ such that
$$
\mathcal{N}(X)\gg (X)^{\frac{K}{(\log \log \log X)^2}}.
$$
\noindent In this paper, we derive a new characterization for the Lucas-Carmichael integers and prove that there are infinitely many Lucas-Carmichael integers assuming the prime $k$-tuples conjecture.

\section{Preliminaries}

\noindent
We start with interesting and elementary results.

\begin{lemma}\label{lemma0.1}
Let $m, n$ be two positive integers with $n > m$. Then $S_{m+1}(n+1) \equiv n+1  \pmod{m}.$
\end{lemma}
\begin{proof}
We write the integer $n+1$ with respect to the base $m+1$ as follows:
\begin{equation}\label{eqn1}
n+1 = n_0 + n_1 (m+1) + n_2 (m+1)^2 + n_3 (m+1)^3 + \cdots,
\end{equation}
where $0 \leq n_i < m+1$ for all $i$.

\medskip
\noindent
Since $(m+1)^k\equiv 1 \pmod{m}$ for all positive integer $k$ and from the equation $(\ref{eqn1}),$ it follows that $S_{m+1}(n+1) \equiv n+1  \pmod{m}.$
\end{proof}

\begin{corollary}\label{cor2.1}
Let $n\geq 1$ be an integer. Then
$$S_{d+1}(n+1)\equiv 1 \pmod{d}$$
for all divisors $d$ of $n.$
\end{corollary}
\begin{proof}
Let $d$ be a divisor of $n$. By Lemma \ref{lemma0.1}, $S_{d+1}(n+1)\equiv n+1 \pmod{d}$, we have $S_{d+1}(n+1) \equiv 1 \pmod{d}.$
\end{proof}

\section{Lucas-Carmichael integers}

In this section, we prove a necessary and sufficient condition for Lucas-Carmichael integers and also prove a few interesting properties of such integers.

\begin{prop}\label{prop3.1}
An integer $n > 1$ is a Lucas-Carmichael integer if and only if $n$ is square-free and $S_{p+2}(n+2)\equiv 1 \pmod{p+1}$ for every prime divisor $p$ of $n$. That is,
$$
\mathcal{L_C}=\Big\{ n\in \mathcal{S}~:~ p~|~n~\Longrightarrow~  S_{p+2}(n+2)\equiv 1 \pmod{p+1}\Big\}.
$$ Here, $\mathcal{L_C}$ and $\mathcal{S}$ denote the set of all Lucas-Carmichael integers and positive square-free integers, respectively.
\end{prop}
\begin{proof}
Let $n > 1$ be a Lucas-Carmichael integer. Clearly, $n$ is square-free, and by Corollary \ref{cor2.1}, $S_{p+2}(n+2)\equiv 1 \pmod{p+1}$
whenever $p$ divides $n.$
\medskip

\noindent
Now, we prove the converse part. Assume that $n$ is a square-free integer satisfying
\begin{equation}\label{prop:eq1}
S_{p+2}(n+2)\equiv 1 \pmod{p+1}
\end{equation}
for all prime divisors $p$ of $n$.

\medskip
\noindent
By Lemma \ref{lemma0.1}, we have
\begin{equation}\label{prop:eq2}
S_{p+2}(n+2) \equiv n+2 \pmod{p+1}.
\end{equation}
\noindent
Combining equations \eqref{prop:eq1} and \eqref{prop:eq2}, it is clear that  $p+1$ divides $n+1$ whenever $p$ divides $n$.
\end{proof}

\begin{prop}\label{prop2}
Every Lucas-Carmichael integer $n$ is odd with at least three prime factors, and $p+1$ $|$ $\frac{n}{p}-1$ for every prime $p$ divides $n.$
\end{prop}
\begin{proof}
Since $p+1$ divides $n+1$, $n+1$ is even, and thus $n$ is odd.

\medskip
\noindent
Suppose that there is a Lucas-Carmichael integer $n$ with exactly two prime factors $p$ and $q.$ Assume that $p>q.$

\medskip
\noindent
Since $p+1$ divides $n+1,$ let $$k := \frac{n+1}{p+1} = \frac{pq+1}{p+1} \in \mathbb{N}.$$
\medskip
\noindent
Then,
\begin{align*}
 k&=\frac{pq-p+p+1}{p+1}\\
 &=\frac{p(q-1)+(p+1)}{p+1}\\
 &= \frac{p(q-1)}{p+1}+1.
\end{align*}
This implies that $p+1$ divides $q-1,$ but it is not possible. Hence, $n$ has at least three prime factors.
\medskip

\noindent
Now, we prove that $p+1$ $|$ $\frac{n}{p}-1$ for every prime $p$ divides $n.$
Let $n=p_1p_2\cdots p_r,$ $r\geq 3$ be a Lucas-Carmichael integer. Since $p_i+1$ divides $n+1$ for all $i$, $$k_i := \frac{n+1}{p_i+1} = \frac{p_1 p_2\cdots p_r+1}{p_i+1} \in \mathbb{N}.$$
Let $n_i=\frac{n}{p_i}$ and we write
\begin{align*}
k_i&=\frac{p_1p_2\cdots p_r+n_i-n_i+1}{p_i+1}\\
&=\frac{n_i(p_i+1)-(n_i - 1)}{p_i+1}\\
&= n_i - \frac{(n_i - 1)}{p_i+1}.
\end{align*}
Therefore, $\frac{(n_i - 1)}{p_i+1}=n_i-k_i \in \mathbb{Z}.$ That is,  $p_i+1$ divides $n_i-1.$ This completes the proof.
\end{proof}

\begin{corollary}
Every prime factor $p$ of a Lucas-Carmichael integer $n$ is strictly less than $\sqrt{n}.$
\end{corollary}
\begin{proof}
By Proposition \ref{prop2}, we have $p < \frac{n}{p}$ for every prime factor $p$ of $n.$ This implies that $p<\sqrt{n}.$
\end{proof}

\begin{prop}
If $n=mqr$ is a Lucas-Carmichael integer where $m \in \mathbb{N}$ and $q,r$ are primes with $q<r.$ Then $q<3m^2$ and $r < 3m^3.$
\end{prop}
\begin{proof}
Since $q$ and $r$ are prime divisors of a Lucas-Carmichael integer $n$, we have $q+1$ $|$ $n+1$ and $r+1$ $|$ $n+1$.

\medskip
\noindent
That is,
$$
mqr \equiv -mr \equiv -1 \pmod{q+1} ~\text{and} ~  mqr \equiv -mq \equiv -1 \pmod{r+1}.
$$
Now, we define
$$C=\frac{mq-1}{r+1} ~~\text{and} ~~ D=\frac{mr-1}{q+1}.
$$
Since $mq-1<mr-1<mr+m,$ we have $C<m.$

\medskip

\noindent
As $r-q\geq 1,$ we have $m-1<m\leq m(r-q).$ This implies that $m+mq<mr+1$. Then
\begin{align*}
mq+m-q-1&<mr+1-q-1\\
m(q+1)-(q+1)& = mr-q < mr-1\\
(m-1)(q+1)&<mr-1\\
m-1&<\frac{mr-1}{q+1}=D.
\end{align*}
Therefore, we have $1\leq C < m \leq D.$ Now we consider:
\begin{align*}
 D(q+1)&=mr-1\\
 &=m\Big(\frac{mq-1}{C}-1\Big)-1\\
 &= \frac{m^2q-m-mC-C}{C}\\
 CD(q+1) &= m^2q-m-mC-C\\
 &= m^2q+m^2-m^2-m-mC-C\\
 (CD-m^2)(q+1)&= -m^2-m-mC-C<0.
\end{align*}
This implies that,
\begin{align*}
 0<(m^2-CD)(q+1)&=m^2+m+mC+C\\
 q+1 &\leq m^2+m(C+1)+C.
\end{align*}
Since $C<m,$ we obtain that $q+1<m^2+m^2+m<3m^2$ and hence $q<3m^2.$
\medskip

\noindent
Next, we prove the other inequality. Consider,
$$r+1=\frac{mq-1}{C}<\frac{m(q+1)}{C} < \frac{m(3m^2))}{C} < 3m^3.$$
and hence the inequality $r < 3m^3$ holds.
\end{proof}

\noindent In the following section, we explicitly describe a class of Lucas-Carmichael integers.

\section{Some general forms of Lucas-Carmichael integers}

Recall that, from Proposition \ref{prop3.1}, an integer $n$ is a Lucas-Carmichael integer if and only if $p+1$ divides $S_{p+2}(n+2) - 1$ whenever $p$ divides $n$. In this section, we study some general forms of Lucas-Carmichael integers with an odd number of prime factors. Also, we define the degree of a Lucas-Carmichael integer and  prove some interesting results on the degree of such integers.

\begin{definition}
An integer $n \in \mathcal{L_C}$ is called a primary Lucas-Carmichael integer if $S_{p+2}(n+2)=p+2$ for every prime $p$ divides $n$, and the set of all such integers is denoted by $\mathcal{L_C}'$.
\end{definition}

\begin{definition}
Let $n$ be a Lucas-Carmichael integer and
$$\alpha :=\max_{p|n}~\Bigg\{\frac{S_{p+2}(n+2) - 1}{p+1}\Bigg\}.$$
The integer $\alpha$ is called the degree of $n$.
\end{definition}

\noindent
We note that primary Lucas-Carmichael integers have a degree of $1$.
\medskip

\noindent
Now, we prove that there are infinitely many Lucas-Carmichael integers assuming the prime $k$-tuples conjecture (defined below).

\medskip

\noindent \textbf{The Prime $k$-tuples Conjecture}. Let $a_1 ,\ldots, a_k$ be positive integers, and let $b_1 ,\ldots, b_k$ be nonzero integers. For $m\geq 1,$ define $f(m)=\prod_{i=1}^k (a_i m + b_i).$ Let $P(x)$ denote the number of positive integers $m \leq x$ for which $a_i m + b_i$ is prime for each $i = 1,\ldots, k.$ The Prime $k$-tuples Conjecture states that if no prime divides $f(m)$ for every $m,$ then there exists $c > 0$ such that $P(x) \sim \frac{cx}{\text{log}^kx}$ as $x\rightarrow \infty.$
\medskip

\noindent Chernick \cite{Chernick} called polynomial of the form $f(m)$ is universal if it satisfies the congruence relations $f(m)\equiv 1 \pmod{a_i m + b_i-1}$ for every $i=1,2,\ldots,k,$ where $m\in \mathbb{Z},$ $k\geq 3$ and $k$ is odd. Further, Chernick \cite{Chernick} proved that for any integers $k \geq 4$ and $m \geq 1$ such that $2^{k-4} $ divides $m,$
$$U_k(m)=(6m+1)(12m+1)\prod_{i=1}^{k-2}(9 \cdot 2^i m + 1)$$
is a Carmichael number if each of the $k$ factors is prime. We call $U_k(m)$ a Chernick polynomial.

\begin{lemma}\label{lemma0.7}
 Let $m$ be a positive integer, and $p = 6m - 1,$ $q = 12m - 1,$ $r = 18m - 1$
and $U_3'(m)=n=pqr.$ Then
\begin{align*}
 n+2&=5(p+2)^3+(p-27)(p+2)^2+45(p+2)+(p-20)\\
 n+2&=(9m-5)(q+2)^2+(9m+11)(q+2)+(6m-5)\\
 n+2&=(4m-2)(r+2)^2+(6m+5)(r+2)+(8m-2).
\end{align*}
\end{lemma}
\begin{proof}
We rewrite $n+2$ with the base $p+2$ as follows:
\begin{align*}
n+2&=(216m^2-102m+23)(p+2)-22\\
&=(36m-23)(p+2)^2+46(p+2)-22\\
&=5(p+2)^3+(p-27)(p+2)^2+46(p+2)-22\\
&=5(p+2)^3+(p-27)(p+2)^2+45(p+2)+(p-20).
\end{align*}
Similarly, we can write $n+2$ with the base $q+2$ as follows:
\begin{align*}
n+2&=(108m^2-42m)(q+2)+(78m+1)\\
&=9m(q+2)^2-5(q+2)^2+(9m+11)(q+2)+(6m-5)\\
&=(9m-5)(q+2)^2+(9m+11)(q+2)+(6m-5).
\end{align*}
Next, we can also write $n+2$ with the base $r+2$ as follows:
\begin{align*}
n+2&=(72m^2-26m)(r+2)+(62m+1)\\
&=4m(r+2)^2-2(r+2)^2+(6m+5)(r+2)+(8m-2)\\
&=(4m-2)(r+2)^2+(6m+5)(r+2)+(8m-2).
\end{align*}
\end{proof}

\begin{corollary}\label{cor4.1}
Assuming the same hypotheses as Lemma $\ref{lemma0.7}$ with $m\geq 8,$ we have $S_{p+2} (n+2) = 2p+3,$ $S_{q+2} (n+2) = 2q+3$ and $S_{r+2} (n+2) = r+2.$
\end{corollary}
\begin{proof}
Since $p\geq 47$ as $m\geq 8,$ the coefficients of $n+2$ in the first equation in Lemma \ref{lemma0.7} are lie between $0$ and $p+1,$ so we have $n+2$ with base $p+2.$ Thus, $S_{p+2}(n+2)=5+(p-27)+45+(p-20)=2p+3.$
\medskip

\noindent
Similary, the coefficients of $n+2=(9m-5)(q+2)^2+(9m+11)(q+2)+(6m-5)$ are between $0$ and $q+1,$ we have $S_{q+2}(n+2)=(9m-5)+(9m+11)+(6m-5)=24m+1=2q+3.$
\medskip

\noindent
Also, the coefficients of $n+2=(4m-2)(r+2)^2+(6m+5)(r+2)+(8m-2)$ are between $0$ and $r+1,$ we have $S_{r+2}(n+2)=(4m-2)+(6m+5)+(8m-2)=18m+1=r+2.$
\end{proof}

\begin{theorem} The prime $k$-tuples conjecture implies that there are infinitely many Lucas-Carmichael integers of degree $2$ with exactly three prime factors.
\end{theorem}
\begin{proof} By the prime $k$-tuples conjecture, there are infinitely many positive integers $m$ such that $U_3'(m)$ has three distinct prime factors. Applying the Corollary $\ref{cor4.1},$ each of these numbers $U_3'(m)$ for $m\geq 8$ is a Lucas-Carmichael integer of degree $2.$
\end{proof}

\begin{theorem}
 Every Lucas-Carmichael integer $n$ with exactly three prime factors is of the form $(2hr_1-1)(2hr_2-1)(2hr_3-1),$ where $h$ is a positive integer and $r_i'$s are pairwise co-prime integers.
\end{theorem}
\begin{proof}
Let $n=p_1p_2p_3$ be a Lucas-Carmichael integer with three prime factors. Write $p_i=r_ik-1,$ where $k$ is the $g.c.d.$ of $p_i+1$ for $i=1,2,3.$ Since $n$ is a Lucas-Carmichael integer,  we have
$$(r_1k-1)(r_2k-1)(r_3k-1)\equiv -1 \pmod{p_i+1}$$ for $i=1,2,3.$ This implies that,
$$(r_1r_2r_3)k^3-(r_1r_2+r_1r_3+r_2r_3)k^2+(r_1+r_2+r_3)k-1\equiv -1 \pmod{r_ik}.$$
By simplifying the above congruence, we obtain
$$-(r_1r_2+r_1r_3+r_2r_3)k+(r_1+r_2+r_3)\equiv 0\pmod{r_i}.$$
For $1\leq i\neq j \leq 3$,  if $r_i$ and $r_j$ have a common factor, then the third one does, and it contradicts our assumption. Hence, the $r_i$'s are pairwise co-prime. Suppose $k$ is odd, then $r_i$'s are even as $p_i$'s are odd. Since  $k$ is the $g.c.d.$ of $p_i+1,$ this is not possible. Therefore, $k$ must be even, and the theorem follows.
\end{proof}

\begin{theorem}\label{thm4.3}
Let $F_l'$ be a Lucas-Carmichael integer with exactly $l$-odd number of prime factors $p_1,p_2,\ldots, p_l.$ Let $k_1$ be the $g.c.d.$ of $p_i+1$ and $r_i=\frac{p_i+1}{k_1}.$ Also, let $R$ be the $l.c.m.$ of $r_i,$ $i=1,2,\ldots,l.$ Then $U_l'(m)=\prod_{i=1}^l(r_iRm+p_i)$ satisfies the congruence relation
 $$U_l'(m)\equiv -1 \pmod{r_iRm+p_i+1}
 $$
 for $i=1,2,\ldots,l.$
\end{theorem}
\begin{proof}
Since $F_l'=\prod_{i=1}^l (r_ik_1-1)$ is a Lucas-Carmichael integer and $l$ is odd,  we have $\prod_{i=1}^l (r_ik_1-1) \equiv -1 \pmod{r_ik_1}$ for $i=1,2,\ldots,l.$

\medskip
\noindent
This implies that
$$
\frac{\prod_{i=1}^l (r_ik_1-1)+1}{k_1} \equiv 0 \pmod{r_i}
$$
for $i=1,2,\ldots,l.$

\noindent
Therefore,
$$
\frac{\prod_{i=1}^l (r_ik_1-1)+1}{k_1} \equiv 0 \pmod{R}.
$$
We observe that, any $k\equiv k_1\pmod{R}$ is also a solution for the above congruence. Write $k=Rm+k_1$ and substitute this $k$ for $k_1,$ we obtain
$$
\frac{\prod_{i=1}^l (r_i(Rm+k_1)-1)+1}{k} \equiv 0 \pmod{R}.
$$
Since $p_i=r_ik-1$ and $r_ik=r_iRm+r_ik_1=r_iRm+p_i+1,$ we have
$$
\prod_{i=1}^l (r_iRm+p_i)+1 \equiv 0 \pmod{r_ik}
$$
and
$$
U_l'(m)=\prod_{i=1}^l (r_iRm+p_i)\equiv -1 \pmod{r_iRm+p_i+1}
$$
for every $i$ ranges over $1$ to $l.$
\end{proof}

\noindent
\textbf{Remark:} Since $U_l'(m)=\prod_{i=1}^l(r_iRm+p_i)$ satisfies the congruence relation
$$
U_l'(m)\equiv -1 \pmod{r_iRm+p_i+1}
$$
for $i=1,2,\ldots,l,$ the integer $U_l'(m)$ is a Lucas-Carmichael integer for every $m$ for which each of the $l$ factors is a prime.

\medskip
\noindent
We illustrate Theorem \ref{thm4.3} with the examples below.

\begin{example}
Let $F_5'= 588455$ be a Lucas-Carmichael integer. Then, we have $k_1=2,$ $r_1=3, r_2=4, r_3=9,r_4=12,r_5=22$ and $R=396.$ Let $U_5'(m)=n=p_1p_2p_3p_4p_5,$ where $p_1 = 1188m+5,$ $p_2 = 1584m+7,$ $p_3 = 3564m+17,$ $p_4=4752m+23$ and $p_5=8712m+43.$ Then, by Theorem $\ref{thm4.3},$ we have $U_5'(m)\equiv -1 \pmod{396r_i+p_i+1}$ for every $i=1,2,3,4,5.$
\end{example}

\noindent
Now, we prove that there are infinitely many Lucas-Carmichael integers of degree $4$ with exactly five prime factors.

\begin{lemma}\label{lemma0.8}
Let $m$ be a positive integer, and $p = 1188m+5,$ $q = 1584m+7,$ $r = 3564m+17,$ $s=4752m+23,$ $t=8712m+43$
and $U_5'(m)=n=pqrst.$ Then
\begin{align*}
n+2&=117(p+2)^5+(396m-875)(p+2)^4+(1056m+2593)(p+2)^3\\
&\quad+(660m-3771)(p+2)^2+(132m+2724)(p+2) +(132m-775)\\
n+2&=27(q+2)^5+\Bigg(\frac{2117016}{1584}m-224\Bigg)(q+2)^4 \\
&\quad  +\Bigg(\frac{1724976}{1584}m+758\Bigg)(q+2)^3  +\Bigg(\frac{2273832}{1584}m-1199\Bigg)(q+2)^2\\
&\quad  +(1287m+958)(q+2) +(1188m-287)\\
n+2&=\Bigg(\frac{6133248}{3564}m+3\Bigg)(r+2)^4+\Bigg(\frac{2317392}{3564}m+32\Bigg)(r+2)^3\\
&\quad +\Bigg(\frac{574992}{3564}m-65\Bigg)(r+2)^2 +\Bigg(\frac{4373424}{3564}m+78\Bigg)(r+2)\\
&\quad +\Bigg(\frac{12005136}{3564}m-11\Bigg)\\
n+2&=\Bigg(\frac{2587464}{4752}m+1\Bigg)(s+2)^4+\Bigg(\frac{3606768}{4752}m+13\Bigg)(s+2)^3\\
&\quad +\Bigg(\frac{14662296}{4752}m-10\Bigg)(s+2)^2 +(1023m+38)(s+2) \\
&\quad  +(4092m+7)\\
n+2&=\Bigg(\frac{419904}{8712}m\Bigg)(t+2)^4+\Bigg(\frac{9191232}{8712}m+6\Bigg)(t+2)^3 \\
&\quad +\Bigg(\frac{41885424}{8712}m+20\Bigg)(t+2)^2  + \Bigg(\frac{31403376}{8712}m+26\Bigg)(t+2) \\
&\quad+\Bigg(\frac{68897952}{8712}m+37\Bigg).\\
\end{align*}
\end{lemma}

\begin{corollary}\label{cor4.2}
Assume that the same hypotheses in Lemma $\ref{lemma0.8}$ with $156816$ $|$ $m.$ Then, we have $S_{p+2} (n+2) = 2p+3,$ $S_{q+2}(n+2) = 4q+5,$ $S_{r+2} (n+2) = 2r+3,$  $S_{s+2} (n+2) = 2s+3$ and $S_{t+2} (n+2) = 2t+3.$
\end{corollary}

\begin{theorem} The prime $k$-tuples conjecture implies that there are infinitely many Lucas-Carmichael integers of degree $4$ with exactly five prime factors.
\end{theorem}
\begin{proof} From the prime $k$-tuples conjecture, we have infinitely many positive integers $m$ which are divisible by $156816,$ and $U_5'(m)$ has exactly five prime factors. By Corollary $\ref{cor4.2},$ each of these numbers $U_5'(m)$ is a Lucas-Carmichael integer of degree $4.$
\end{proof}

\begin{example}
Let $F_7'=3512071871$ be a Lucas-Carmichael integer. Then, we have $k_1=2,$ $r_1=4, r_2=6, r_3=9,r_4=12,r_5=16, r_6=27,r_7=36$ and $R=432.$ Let $U_7'(m)=p_1p_2p_3p_4p_5p_6p_7,$ where $p_1 = 1728m + 7,$ $p_2 = 2592m + 11,$ $p_3 = 3888m + 17,$ $p_4=5184m + 23,$ $p_5=6912m + 31,$ $p_6=11664m + 53,$ $p_7 = 15552m + 71$
and $U_7'(m)=n=p_1p_2p_3p_4p_5p_6p_7.$ Then, by Theorem $\ref{thm4.3},$ we have $U_7'(m)\equiv -1 \pmod{432r_i+p_i+1}$ for every $i=1,2,3,4,5,6,7.$
\end{example}

\noindent Next, we prove that there are infinitely many Lucas-Carmichael integers of degree $4$ with exactly seven prime factors.

\begin{lemma}\label{lemma0.9}
 Let $m$ be a positive integer, and $p = 1728m + 7,$ $q = 2592m + 11,$ $r = 3888m + 17,$ $s=5184m + 23,$ $t=6912m + 31,$ $u=11664m + 53,$ $v = 15552m + 71$
and $U_7'(m)=n=pqrstuv.$ Then
\begin{align*}
n+2 &= 2460 (p+2)^7 +(648m - 24487)(p+2)^6  \\
&\quad + (918m + 103732) (p+2)^5 + (1512m - 242426)(p+2)^4  \\
&\quad  + (1026m + 337790) (p+2)^3 + (432m - 280653)(p+2)^2  \\
&\quad  + (1512m + 128798)(p+2)  + (864m - 25181)\\
n+2 &=  143(q+2)^7 +(2592m-1644)(q+2)^6 + 7965(q+2)^5  \\
&\quad + ( 1728m - 21163)(q+2)^4  + ( 720m + 33401) (q+2)^3  \\
&\quad + (216m - 31287)(q+2)^2  + ( 2088m + 16136)(q+2) \\
&\quad + (432m - 3525)\\
n+2 &=  8(r + 2)^7 + (1664m - 107)(r + 2)^6 + \left(\frac{1456}{3}m + 662\right)(r + 2)^5\\
&\quad  + \left(\frac{944}{3}m - 2064\right)(r + 2)^4 + \left(\frac{11152}{3}m + 3829\right)(r + 2)^3 \\
&\quad + \left(\frac{3776}{3}m-4148\right)(r + 2)^2+ \left(\frac{3376}{3}m + 2482\right)(r + 2) \\
&\quad + \left(\frac{9296}{3}m - 607\right)\\
\end{align*}
\begin{align*}
n+2 &= (s + 2)^7 + (648m - 15)(s + 2)^6 + (1458m + 125)(s + 2)^5\\
&\quad + (1404m - 419)(s + 2)^4 + (4086m + 915)(s + 2)^3 \\
&\quad + (4212m - 1090)(s + 2)^2 + (2016m + 758)(s + 2) \\
&\quad + (1728m - 202) \\
\end{align*}
\begin{align*}
n+2 &= \left(\frac{531441}{512}m + 2\right)(t + 2)^6 + \left(\frac{460701}{512}m + 26\right)(t + 2)^5 \\
&\quad + \left(\frac{312795}{128}m - 82\right)(t + 2)^4 + \left(\frac{1532709}{256}m + 256\right)(t + 2)^3 \\
&\quad + \left(\frac{811161}{512}m - 319\right)(t + 2)^2 + \left(\frac{1590921}{512}m + 267\right)(t + 2) \\
&\quad + \left(\frac{1453005}{256}m - 53\right)
\\
n+2 &=  \left(\frac{32768}{729}\right)m(u + 2)^6 + \left(\frac{919552}{729}m + 7\right)(u + 2)^5 \\
&\quad + \left(\frac{793744}{729}m - 2\right)(u + 2)^4 + \left(\frac{1154864}{243}m + 45\right)(u + 2)^3 \\
&\quad + \left(\frac{1404608}{729}m - 36\right)(u + 2)^2 + \left(\frac{5584880}{729}m + 82\right)(u + 2) \\
&\quad + \left(\frac{4805968}{729}m + 13\right)\\
n+2 &= 8m(v + 2)^6 + \left(\frac{946}{3}m + 1\right)(v + 2)^5  + \left(\frac{33416}{3}m + 50\right)(v + 2)^4 \\
&\quad + \left(\frac{26998}{3}m + 49\right)(v + 2)^3 + \left(\frac{13664}{3}m + 4\right)(v + 2)^2 \\
&\quad + \left(\frac{31096}{3}m + 69\right)(v + 2) + \left(\frac{33824}{3}m + 44\right).
\end{align*}
\end{lemma}
\begin{corollary}\label{cor4.3}
Assume that the same hypotheses in Lemma $\ref{lemma0.9}$ with $ 373248$ $|$ $m.$ Then, we have $S_{p+2} (n+2) = 4p+5,$ $S_{q+2}(n+2) = 3q + 4,$ $S_{r+2} (n+2) = 3r + 4 ,$  $S_{s+2} (n+2) = 3s + 4$ $S_{t+2} (n+2) = 3t + 4,$ $S_{u+2} (n+2) = 2u + 3$ and $S_{v+2} (n+2) = 3v+ 4.$
\end{corollary}

\begin{theorem} The prime $k$-tuples conjecture implies that there are infinitely many Lucas-Carmichael integers of degree $4$ with exactly seven prime factors.
\end{theorem}
\begin{proof} According to the prime $k$-tuples conjecture, there are infinitely many positive integers $m$ divisible by $ 373248$ for which $U_7'(m)$ has exactly seven prime factors. By Corollary $\ref{cor4.3},$ each of these numbers $U_7'(m)$ is a Lucas-Carmichael integer of degree $4.$
\end{proof}

\medskip
\noindent
\textbf{Acknowledgements.} The second author would like to acknowledge support from MeitY QCAL. Also, the second author would like to acknowledge support from ICTP through the associate programme.

%The authors would like to thank Samuel S. Wagstaff, Jr and B. C. Kellner for their valuable comments.

\end{document}